\RequirePackage{fix-cm}
\documentclass[smallextended,envcountsect]{svjour3}      
\smartqed  
\usepackage{graphicx}
\usepackage{amsmath,amsfonts}
\usepackage{amssymb}
\usepackage{cases}
\usepackage[colorlinks,linkcolor=blue,anchorcolor=blue,citecolor=blue,urlcolor=blue]{hyperref}
\newcommand{\dom}{\text{dom}}
\newcommand{\epi}{\text{epi}}
\newcommand{\C}{\mathcal{C}}
\newcommand{\R}{\mathbb{R}}

\DeclareMathOperator*{\argmin}{argmin}
\newtheorem{assumption}{Assumption}

\usepackage[ruled,linesnumbered]{algorithm2e}
\usepackage{listings}
\begin{document}

\title{PGA-based Predictor-Corrector Algorithms for Monotone Generalized Variational Inequality
}

\titlerunning{PGA-based Predictor-Corrector Algorithms for MGVI}        

\author{Yu You
}


\institute{Yu You \at
              Shanghai Jiao Tong University, School of Mathematical Sciences, China \\
              \email{youyu0828@sjtu.edu.cn}           
}

\date{Received: date / Accepted: date}

\maketitle

\begin{abstract}
In this paper, we consider the  monotone generalized variational inequality (MGVI) where the monotone operator is Lipschitz continuous. Inspired by the extragradient method \cite{MR451121} and the projection contraction algorithms \cite{MR1418264,MR1297058} for monotone variational inequality (MVI), we propose a class of PGA-based Predictor-Corrector algorithms for MGVI. A significant characteristic of our algorithms for separable multi-blocks convex optimization problems is that they can be well adapted for parallel computation. Numerical simulations about different models for sparsity recovery show the wide applicability and effectiveness of our proposed methods.
\end{abstract}

\section{Introduction}
\label{intro}
Let $\theta :\R^n\rightarrow (-\infty,\infty]$ be a closed proper convex function, and $F:\R^n\rightarrow \R^n$ a vector-valued and continuous mapping. The generalized variational inequality (GVI) \cite{MR686212,MR1673631} takes the next form 
\begin{equation*}
\label{GVI}
x^* \in \R^n, ~\theta(x)-\theta(x^*) + (x-x^*)^TF(x^*) \geq 0, ~\forall x \in \R^n~.
\end{equation*}
Here, we keep the arithmetic operation that $\infty - \infty = \infty$. In this paper, we focus on the case where $F$ is monotone, i.e., $ (F(x)-F(y))^T(x-y) \geq 0$ for all $x,y\in \R^n$. In this case, the corresponding \text{GVI} is called monotone generalized variational inequality (\text{MGVI}). Throughout this paper, we make the following additional assumptions for $\text{MGVI}(\theta,F)$:
\begin{assumption}
	\rm
	\label{assum:1}
	\begin{enumerate}
		\item[(a)] The monotone operator $F$ is Lipschitz continuous with Lipschitz constant $L>0$.
		\item[(b)] The solution set of $\text{MGVI}(\theta,F)$, denoted as $\Omega^*$, is nonempty.
	\end{enumerate}
\end{assumption} 
Various convex optimization problems can be formulated as MGVIs, such as lasso problem, basis pursuit problem \cite{MR1854649}, basis pursuit denoising problem \cite{MR2481111}, and Dantzig selector \cite{MR2382651}, to name a few. Moreover, MGVI contains the classical monotone variational inequality (MVI). One sees this by setting $\theta = \delta_{\C}$\footnote{$\delta_{\C}$ is the indicator function of $\C$ which is equal to $0$ if $x\in\C$ and $\infty$ otherwise.} with $\C \subseteq \R^n$ being a nonempty closed convex set, then MGVI reduces to MVI in form of 
\begin{equation*}
\label{MVI}
x^* \in \C,~ (x-x^*)^TF(x^*) \geq 0,~ \forall x \in \C~.
\end{equation*}
For solving $\text{MVI}(\theta,F,\C)$, the extragradient method \cite{MR451121} can be applied. Specifically, at iteration $k$, in predictor step, the projection operator is used for generating a predictor $\tilde{x}^k = \text{P}_{\C}(x^k-\beta^k F(x^k))$ where $\beta^k$ is selected to satisfy some given condition, then in corrector step, $x^{k+1} = \text{P}_{\C}(x^k-\beta^k F(\tilde{x}^k))$. This method yields a sequence converging to a solution of $\text{MVI}(\theta,F,\C)$.  On the other hand, He \cite{MR1418264,MR1297058} proposed a class of projection and contraction methods for $\text{MVI}(\theta,F,\C)$, which also belongs to a class of Predictor-Corrector algorithms. In predictor step, the predictor is the same with the above extragradient method as  $\tilde{x}^k = \text{P}_{\C}(x^k-\beta^k F(x^k))$, then the corrector step is based on the next three fundamental inequalities for constructing different profitable directions:
\begin{equation}
\label{eq:1_1}
(\tilde{x}^k-x^*)^T\beta^k F(x^*) \geq 0, 
\end{equation}
\vspace{-0.5cm}
\begin{equation}
\label{eq:1_2}
(\tilde{x}^k-x^*)^T([x^k-\beta^k F(x^k)] -\tilde{x}^k) \geq 0,
\end{equation}
\vspace{-0.5cm}
\begin{align}
\label{eq:1_3}
(\tilde{x}^k-x^*)^T(\beta^k F(\tilde{x}^k)-\beta^k F(x^*)) \geq 0,
\end{align}
where $x^*$ is any solution of $\text{MVI}(\theta,F,\C)$, and $\beta^k >0$.
Numerical simulations therein \cite{MR1418264,MR1297058} show that the methods proposed by He outperformed the extragradient method in general. 

Inspired by both of their works for $\text{MVI}(\theta,F,\C)$, we generalize their ideas for $\text{MGVI}(\theta,F)$. Firstly, for generalizing the idea of extragradient method for $\text{MGVI}(\theta,F)$, at iteration $k$ and in predictor step, we use the proximity operator to yield a predictor $\tilde{x}^k = \text{Prox}_{\beta^k \theta}(x^k-\beta^kF(x^k))$, where $\beta^k$ is selected to meet some given condition, then in corrector step, we set $x^{k+1}$ as $\text{Prox}_{\beta^k \theta}(x^k-\beta^kF(\tilde{x}^k))$. The resulting algorithm can be deemed as a natural generalization of extragradient method (EM) for $\text{MGVI}(\theta,F)$. Thus we denote our algorithm as GEM. The difference between GEM and EM is the different operator used in both predictor and corrector steps. Surprisingly, this kind of modification does not change the key inequality for the convergence analysis of GEM compared with that for EM.
Secondly, for generalizing the idea of He's projection and contraction algorithms, we also use the proximity operator to make a predictor $\tilde{x}^k =\text{Prox}_{\beta^k \theta}(x^k-\beta^kF(x^k))$, then we likewise obtain three fundamental inequalities as follows:
\begin{equation}
\label{eq:1_4}
\beta^k\theta(\tilde{x}^k) - \beta^k \theta(x^*) + (\tilde{x}^k-x^*)^T\beta^k F(x^*) \geq 0, 
\end{equation}
\begin{equation}
\label{eq:1_5}
\beta^k \theta(x^*) -\beta^k\theta(\tilde{x}^k) + (\tilde{x}^k-x^*)^T([x^k-\beta^k F(x^k)] -\tilde{x}^k) \geq 0,
\end{equation}
\begin{align}
\label{eq:1_6}
(\tilde{x}^k-x^*)^T(\beta^k F(\tilde{x}^k)-\beta^k F(x^*)) \geq 0,
\end{align} 
where $x^*$ is any solution of $\text{MGVI}(\theta,F)$, and $\beta^k >0$. We observe that the leftmost terms of (\ref{eq:1_4}) and (\ref{eq:1_5}) are opposite to each other, thus they can be eliminated when adding (\ref{eq:1_4}) and (\ref{eq:1_5}), or when adding (\ref{eq:1_4}), (\ref{eq:1_5}) and (\ref{eq:1_6}). This yields the same results with that when adding (\ref{eq:1_1}) and (\ref{eq:1_2}), or adding (\ref{eq:1_1}), (\ref{eq:1_2}), and (\ref{eq:1_3}). The latter is used by He in corrector step for constructing projection and contraction algorithms for $\text{MVI}(\theta,F,\C)$. Thus, the methodologies developed by He for $\text{MVI}(\theta,F,\C)$ can be straightforwardly inherited. This is an inspiring result. The only difference between He's works for $\text{MVI}(\theta,F,\C)$ and our proposed algorithms for $\text{MGVI}(\theta,F)$ will be the different predictors, which differ a little. Here, it should be mentioned that He et al. \cite{MR2902659,MR3049509,MR3128749,MR2500836} developed diverse PPA-based contraction methods for $\text{MGVI}(\theta,F)$, one can also refer to \cite{MR3618786} for a review about the works of He et al. in $\text{MVI}(\theta,F,\C)$ and $\text{MGVI}(\theta,F)$. Although their PPA-based contraction algorithms possess the contraction properties, the various corrector strategies developed by them for  $\text{MVI}(\theta,F,\C)$ are not used. Thus our proposed algorithms for $\text{MGVI}(\theta,F)$ differ from theirs. Part of our contribution in this paper is to construct a  direct link between their previous works in $\text{MVI}(\theta,F,\C)$ and later works in $\text{MGVI}(\theta,F)$. The above two generalizations can be treated as Predictor-Corrector algorithms, and both of them use the proximity operator for making predictors, thus we call our algorithms as PGA-based Predictor-Corrector algorithms. 

Note that in order to make our algorithms to be implementable in practice, the proximity operator of $\theta$ should be easy to compute. The reader can refer to \cite{MR3616647,MR3719240,MR2858838} for a wide variety of computable proximity operators. Finally, we mention that for the structured convex optimization problem with a separable objective function and linear constraint, our proposed algorithms can be well adapted for parallel computation, which can not be realized by many algorithms related to ADMM. This is a significant characteristic of our algorithms, allowing lesser time exhausted by employing the technique of parallel computation, especially in case of large-scale multi-blocks separable problems.

The rest of this paper is organized as follows.  In section \ref{sec:2}, we recall some notations and preliminaries. Then in section \ref{sec:3}, we will introduce the first class of  our proposed PGA-based Predictor-Corrector algorithms, which is a natural generalization of extragradient method. Next in section \ref{sec:4}, the second class of our algorithms will be introduced. In section \ref{sec:5}, we will discuss and compare our proposed algorithms with ADMM-related algorithms with respect to a concrete convex optimization problem. Finally in section \ref{sec:6}, we test our proposed algorithms for different sparsity recovery models.

\section{Notations and preliminaries}
\label{sec:2}
Let $\R^n$ denote the $n$-dimensional Euclidean space with standard inner product $\langle \cdot, \cdot \rangle$, i.e., $\langle x, y \rangle = x^Ty$ for any $x,y\in \R^n$, and $\Vert \cdot \Vert$ denote the Euclidean norm. Given a function $f:\R^n \rightarrow (-\infty,\infty]$, the set  $$\dom f := \{x\in \R^n:f(x) < \infty \}$$ denotes its effective domain. If $\dom f \neq \emptyset$, then $f$ is called a proper function (or proper). The set
$$\epi f:=\{(x,t):f(x)\leq t,x\in \R^n,t\in \R\}$$ denotes the epigraph of $f$, and $f$ is closed (resp. convex) if $\epi f$ is closed (resp. convex). The function $f$ is called a closed proper convex function if it is proper, closed and convex.

Now, let $f:\R^n\rightarrow(-\infty,\infty]$ be a closed proper convex function, the proximity operator of $f$ is defined as follows: \[\text{Prox}_f :x \mapsto \argmin\{f(y)+\frac{1}{2}\Vert x-y\Vert^2:y\in\R^n\},\] the optimal solution is unique. If $f = \delta_{\C}$ with $\C$ being a nonempty closed convex set, then the above proximity operator reduces to the usual projection operator $\text{P}_{\C}$. 

\begin{lemma}[\cite{MR3719240}]
	\label{lem:1}
	Let $f:\R^n \rightarrow (-\infty,\infty]$ be a closed proper convex function. Then the next three conditions are equivalent:
	\begin{enumerate}
		\item $p = \text{Prox}_f(x)$.
		\item $x-p \in \partial f(p)$.
		\item For all $w \in \R^n$, it holds that 
		\begin{equation*}
		(x-p)^T(p-w) \geq f(p) - f(w).
		\end{equation*}
		
	\end{enumerate}
\end{lemma}

For any $\beta > 0$, let $e(x,\beta):= x -\text{Prox}_{\beta\theta}(x-\beta F(x))$. The next theorem shows that $x^*\in \R^n$ is a solution of $\text{MGVI}(\theta,F)$ if and only if $e(x,\beta)=0$. 
\begin{lemma}
	For any $\beta > 0$, $x^*\in \R^n$ is a solution of $\text{MGVI}(\theta,F)$ if and only if $e(x,\beta)=0$. 
\end{lemma}
\begin{proof}
	If $x^*$ is a solution of $\text{MGVI}(\theta,F)$, then the definition of $\text{MGVI}(\theta,F)$ implies $-F(x^*) \in \partial \theta(x^*)$, that is same with $[x^* -\beta F(x^*)]-x^* \in \partial \beta \theta (x^*)$, then Lemma \ref{lem:1} implies that 
	\[\text{Prox}_{\beta \theta}(x^*-\beta F(x^*)) = x^*,\]
	thus $e(x,\beta) = 0$. Now, let $e(x,\beta) = 0$, this implies $\text{Prox}_{\beta \theta}(x^*-\beta F(x^*)) = x^*$, then it follows from Lemma \ref{lem:1} that $-F(x^*) \in \partial \theta(x^*)$, thus $x^*$ is a solution of $\text{MGVI}(\theta,F)$. \qed
\end{proof}
Thus, solving $\text{MGVI}(\theta,F)$ reduces to solving a zero point of $e(x,\beta)$. So, for given $\beta > 0$, $\Vert e(x,\beta)\Vert$ can be treated as a error measure function. Next we will show that  $\Vert e(x,\beta)\Vert$ is a non-decreasing function about $\beta$, while $\Vert e(x,\beta)\Vert/\beta$ is a non-increasing function about $\beta$. This show that $e(x,\beta)$ can be used for stopping criterion when implementing algorithms. The next result can be seen by a simple modification of notations in the proof of Theorem 10.9 in \cite{MR3719240}. 


\begin{lemma}[\cite{MR3719240}]
	For all $x\in \R^n$, and $\beta_2 \geq \beta_1 >0$, it holds that 
	\begin{equation}
	\label{eq:2_1}
	\Vert e(x,\beta_2)\Vert \geq 	\Vert e(x,\beta_1)\Vert
	\end{equation}
	and 
	\begin{equation}
	\label{eq:2_2}
	\frac{\Vert e(x,\beta_2)\Vert}{\beta_2} \leq 	\frac{\Vert e(x,\beta_1)\Vert}{\beta_1}.
	\end{equation}
\end{lemma}

Before ending this part, for simplifying the convergence analysis of our proposed algorithms, the definition of Fej\'{e}r monotone sequence and its related properties are recalled.

\begin{definition}[Fej\'{e}r monotone sequence] 
	Let $\mathcal{D} \subseteq \R^n$ be nonempty, and let $\{x^k\}_k$ be a sequence in $\R^n$, then $\{x_k\}_k$ is Fej\'{e}r monotone with respect to $\mathcal{D}$ if for all $k$ and any $x\in \mathcal{D}$, it holds that
	$$ \Vert x^{k+1}-x\Vert \leq \Vert x^k -x\Vert.$$ 
\end{definition}

\begin{lemma}[\cite{MR3616647}]
	\label{lem:4}
	Let $\mathcal{D} \subseteq \R^n$ be nonempty, and let $\{x^k\}_k$ be a Fej\'{e}r monotone sequence with respect to $\mathcal{D}$, then the following hold
	\begin{itemize}
		\item $\{x^k\}_k$ is bounded;
		\item If any subsequence of $\{x^k\}_k$ converges to a point in $\mathcal{D}$, then $\{x^k\}_k$ converges to a point in $\mathcal{D}$. 
	\end{itemize}
\end{lemma}

\section{Generalized Extragradient Method for $\text{MGVI}(\theta,F)$}
\label{sec:3}
In this section, we will introduce the first class of our PGA-based Predictor-Corrector algorithms for solving $\text{MGVI}(\theta,F)$. The motivation for our algorithm is the extragradient method (EM) \cite{MR451121}, which is designed for solving $\text{MVI}(\theta,F,\C)$. Our algorithm for $\text{MGVI}(\theta,F)$ can be deemed as a natural generalization of EM, and is denoted by us as GEM. The difference of GEM from EM is that the proximity operator is used for both of the predictor and corrector steps, while projector operator is used in EM. Next, we will show firstly GEM for $\text{MGVI}(\theta,F)$ in Algorithm \ref{algo:GEM}, followed by the its convergence. 
\begin{algorithm}[ht]
	\caption{GEM for $\text{MGVI}(\theta,F)$}
	\label{algo:GEM}
	\KwIn{$x^0 \in \R^n$, and $\nu \in (0,1)$.}
	
	\For{$k=0, 1, \cdots$}
	{   
		\textbf{Predictor:} Selecting $\beta^k >0$ such that $\beta^k\Vert F(x^k)-F(\tilde{x}^k)\Vert \leq \nu \Vert x^k-\tilde{x}^k\Vert$ where $\tilde{x}^k = \text{Prox}_{\beta^k \theta}(x^k-\beta^k F(x^k))$ ;
		
		\textbf{Corrector:} Setting $x^{k+1} = \text{Prox}_{\beta^k \theta}(x^k-\beta^k F(\tilde{x}^k))$.		
	}
	
\end{algorithm} 

\begin{remark}
	Under the second term of Assumption \ref{assum:1}, we conclude that $$\beta^k\Vert F(x^k)-F(\tilde{x}^k)\Vert \leq \nu \Vert x^k-\tilde{x}^k\Vert$$ whenever $\beta^k \leq \frac{\nu}{L}$.
\end{remark}

In order to prove the convergence of GEM, we need the next propositions.
\begin{proposition}
	\label{prop1}
	Let $x, \tilde{x}, y \in \R^n$, $\beta >0$, and let $$\widetilde{T}_{\beta}(x) := \text{Prox}_{\beta \theta}(x-\beta F(\tilde{x})),$$ then the following inequality holds 
	\[\Vert \widetilde{T}_{\beta}(x) -y\Vert^2\leq \Vert x-y\Vert^2 -\Vert x-\widetilde{T}_{\beta}(x)\Vert^2 + 2\beta(y-\widetilde{T}_{\beta}(x))^T F(\tilde{x}) + 2\beta( \theta(y)- \theta(\widetilde{T}_{\beta}(x))).   \]
\end{proposition}
\begin{proof}
	Consider the function
	\[\phi(u) = (u-x)^TF(\tilde{x}) + \theta(u)+\frac{1}{2\beta}\Vert u-x\Vert^2.\] 
	Because $\widetilde{T}_{\beta}(x) = \argmin\{\phi(u):u\in \R^n\}$, and $\phi$ is a  $\frac{1}{\beta}$-strongly closed proper convex function, then it follows from Theorem 5.25 in \cite{MR3719240} that 
	\begin{equation}
	\label{prop:eq1}
	\phi(y)-\phi(\widetilde{T}_{\beta}(x)) \geq \frac{1}{2\beta}\Vert y-\widetilde{T}_{\beta}(x)\Vert^2,
	\end{equation}
	on the other hand, one obtains that
	\begin{equation}
	\label{prop:eq2}
	\begin{array}{ll}
	\phi(y)-\phi(\widetilde{T}_{\beta}(x)) &= (y-\widetilde{T}_{\beta}(x))^TF(\tilde{x}) + \theta(y)-\theta(\widetilde{T}_{\beta}(x)) \\ &\quad +\frac{1}{2\beta}(\Vert x-y\Vert^2-\Vert x-\widetilde{T}_{\beta}(x)\Vert^2).
	\end{array}
	\end{equation} 
	Combining (\ref{prop:eq1}) and (\ref{prop:eq2}), we obtain the targeted inequality
	\[\Vert \widetilde{T}_{\beta}(x) -y\Vert^2\leq \Vert x-y\Vert^2 -\Vert x-\widetilde{T}_{\beta}(x)\Vert^2 + 2\beta(y-\widetilde{T}_{\beta}(x))^T F(\tilde{x}) + 2\beta( \theta(y)- \theta(\widetilde{T}_{\beta}(x))).  \] 
	\qed
\end{proof}

Next, we prove the key inequality for the convergence analysis of GEM, which is based on proposition \ref{prop1}.
\begin{proposition}
	\label{prop2}
Let $\{x^k\}_k$ be the sequence generated by GEM (Algorithm \ref{algo:GEM}), and let $x^*$ be any solution of $\text{MGVI}(\theta,F)$, then the next inequality holds
\begin{equation*}
\Vert x^{k+1}-x^*\Vert^2 \leq \Vert x^k-x^*\Vert^2 - (1-\nu^2)\Vert x^k-\tilde{x}^k\Vert^2.
\end{equation*}	
\end{proposition}
\begin{proof}
	Substituting $x^*=y$, $x^k=x$, $\tilde{x}^k = \tilde{x}$, $\beta^k = \beta$, and $x^{k+1}=\widetilde{T}_{\beta^k}(x^k)$ in the inequality of proposition \ref{prop1}, we obtain 
	\begin{equation}
	\label{prop:eq3}
	\begin{array}{ll}
	\Vert x^{k+1} -x^*\Vert^2 \leq& \Vert x^k-x^*\Vert^2 -\Vert x^k-x^{k+1}\Vert^2 \\
	&+ 2\beta^k(x^*-x^{k+1})^T F(\tilde{x}^k) + 2\beta^k( \theta(x^*)- \theta(x^{k+1})).
	\end{array}
	\end{equation}
	On the other hand, $-F(x^*)\in \partial \theta(x^*)$ implies that
	\begin{equation}
	\label{prop:eq4}
	2\beta^k \theta(\tilde{x}^k)-2\beta^k\theta(x^*)+2\beta^k(\tilde{x}^k-x^*)^TF(\tilde{x}^k) \geq 0.
	\end{equation}
	Adding inequality (\ref{prop:eq4}) to the right side of inequality (\ref{prop:eq3}), then it follows that
	\begin{equation}
	\label{prop:eq5}
	\begin{array}{ll}
		\Vert x^{k+1} -x^*\Vert^2 &\leq \Vert x^k-x^*\Vert^2 -\Vert x^k-x^{k+1}\Vert^2\\
		&\quad + 2\beta^k(\tilde{x}^k-x^{k+1})^T F(\tilde{x}^k) +2\beta^k( \theta(\tilde{x}^k)- \theta(x^{k+1})) \\
		 &=\Vert x^k-x^*\Vert^2-\Vert x^k-\tilde{x}^k\Vert^2 - \Vert \tilde{x}^k-x^{k+1}\Vert^2 \\
		&\quad -2(x^{k+1}-\tilde{x}^k)^T(\tilde{x}^k-x^k+\beta^kF(\tilde{x}^k)) + 2\beta^k(\theta(\tilde{x}^k)-\theta(x^{k+1}),
	\end{array}
	\end{equation}
	where the second equality follows from $$\Vert x^k-x^{k+1}\Vert^2 = \Vert x^k-\tilde{x}^k\Vert^2 + \Vert \tilde{x}^k-x^{k+1}\Vert^2 + 2(x^k-\tilde{x}^k)^T(\tilde{x}^k-x^{k+1}).$$
	Moreover, $\tilde{x}^k = \text{Prox}_{\beta^k \theta}(x^k-\beta^k F(x^k))$ (Lemma \ref{lem:4}) implies that 
	\begin{equation}
	\label{prop:eq6}
	2\beta^k\theta(x^{k+1})-2\beta^k \theta(\tilde{x}^k) + 2(\tilde{x}^k-x^{k+1})^T(x^k-\beta^kF(x^k)-\tilde{x}^k) \geq 0.
	\end{equation} 
	Then adding (\ref{prop:eq6}) to the right side of (\ref{prop:eq5}), it follows that
	\begin{equation}
	\label{prop:eq7}
	\begin{array}{ll}
	\Vert x^{k+1}-x^*\Vert^2 &\leq \Vert x^k-x^*\Vert^2-\Vert x^k-\tilde{x}^k\Vert^2 - \Vert \tilde{x}^k-x^{k+1}\Vert^2 \\
	&\quad + 2\beta^k(\tilde{x}^k-x^{k+1})^T(F(\tilde{x}^k)-F(x^k)).
	\end{array}
	\end{equation}
	Considering also that
	\begin{equation}
	\label{prop:eq8}
	2\beta^k(\tilde{x}^k-x^{k+1})^T(F(\tilde{x}^k)-F(x^k)) \leq \Vert \tilde{x}^k-x^{k+1}\Vert^2 + \Vert\beta^k( F(\tilde{x}^k)-F(x^k))\Vert^2,
	\end{equation}
	then combining (\ref{prop:eq7}) and (\ref{prop:eq8}), and also taking into account that $$\Vert\beta^k( F(\tilde{x}^k)-F(x^k))\Vert^2 \leq \nu^2 \Vert \tilde{x}^k-x^k\Vert^2,$$
	we conclude that  
	\begin{equation*}
	\Vert x^{k+1}-x^*\Vert^2 \leq \Vert x^k-x^*\Vert^2 - (1-\nu^2)\Vert x^k-\tilde{x}^k\Vert^2.
	\end{equation*}	
	\qed
\end{proof}

\begin{theorem}[Convergence of \text{GEM} for $\text{MGVI}(\theta,F)$]
	\label{them:1}
	Let $\{x^k\}_k$ be the sequence generated by GEM, and let $\{\beta^k\}_k$ be bounded such that $\inf \{\beta^k\} >0$. Then $\{x^k\}_k$  converges to a solution of $\text{MGVI}(\theta,F)$. 
\end{theorem}
\begin{proof}
	One sees from proposition \ref{prop2} that for all $k$, and any $x^* \in \Omega^*$, it holds that
	\[\Vert x^{k+1}-x^*\Vert^2 \leq \Vert x^k-x^*\Vert^2 - (1-\nu^2)\Vert x^k-\tilde{x}^k\Vert^2.\]
	This implies that $\{x^k\}_k$ is Fej\'{e}r monotone with respect to $\Omega^*$. Adding the above inequality from $0$ to $n$, taking some rearrangements, and letting $n\rightarrow \infty$, then it holds that 
	\begin{equation*}
	\sum_{k=0}^{\infty}(1-\nu^2)\Vert x^k-\tilde{x}^k\Vert^2 \leq \Vert x_0 - x^*\Vert^2.
	\end{equation*} 
	Thus $\sum_{k=0}^{\infty}\Vert x^k-\tilde{x}^k\Vert^2 <\infty$, this implies that $\Vert x^k - \tilde{x}^k\Vert \rightarrow 0$, from which we conclude $\{x^k\}_k$ and $\{\tilde{x}^k\}_k$ have common limit points. Let $y$ be any of their common limit points, then $x^{k_j} \rightarrow y$ and $\tilde{x}^{k_j}\rightarrow y$. Because $\{\beta^k\}_k$ is bounded and $\inf\{\beta_k\}>0$, without loss of generality, one can assume that $\beta^{k_j}\rightarrow \beta^*$ with $\beta^*>0$.  Clearly, $\tilde{x}^{k_j} = \text{Prox}_{\beta^{k_j} \theta}(x^{k_j} -\beta^{k_j}F(x^{k_j}))$, then it follows from Lemma \ref{lem:1} that for any $x\in \R^n$, it holds 
	\begin{equation*}
	\beta^{k_j} \theta(x) - \beta^{k_j}\theta(\tilde{x}^{k_j}) + (x-\tilde{x}^{k_j})^T(\tilde{x}^{k_j}-[x^{k_j}-\beta^{k_j}F(x^{k_j})]) \geq 0,
	\end{equation*}
	thus
	\begin{align}
	\label{them1:eq1}
	&\limsup \{\beta^{k_j} \theta(x) - \beta^{k_j}\theta(\tilde{x}^{k_j}) + (x-\tilde{x}^{k_j})^T(\tilde{x}^{k_j}-[x^{k_j}-\beta^{k_j}F(x^{k_j})])\} \nonumber\\
	=&\beta^*\theta(x)-\liminf \beta^{k_j}\theta(\tilde{x}^{k_j}) + (x-y)^T(\beta^*F(y)) \geq 0.
	\end{align}
	On the other hand, because $\theta$ is a closed proper convex function, this implies that $\theta$ is lower semicontinuous, thus $\theta(y) \leq \liminf \theta(\tilde{x}^{k_j})$. Thus
	\begin{equation}
	\label{them1:eq2}
	-\liminf \beta^{k_j}\theta(\tilde{x}^{k_j}) \leq -\beta^*\theta(y).
	\end{equation}
	Then combining (\ref{them1:eq1}) and (\ref{them1:eq2}), and considering that $\beta^* >0$, we conclude that
	
	\begin{equation*}
	\theta(x) -\theta(y) + (x-y)^TF(y) \geq 0,
	\end{equation*}
	for any $x\in\R^n$.
	Thus $y$ is a solution of $\text{MGVI}(\theta,F)$, i.e., $y\in \Omega^*$. Finally, because $\{x^k\}_k$ is Fej\'{e}r monotone with respect to $\Omega^*$, $y$ is any limit point of $\{x^k\}_k$, and $y\in \Omega^*$, thus Lemma \ref{lem:4} implies that $\{x^k\}_k$ converges to a solution of $\text{MGVI}(\theta,F)$. \qed
\end{proof}

\begin{remark}
To guarantee the convergence of GEM (Algorithm \ref{algo:GEM}), $\{\beta^k\}_k$ requires to be bounded with $\inf\{\beta^k\} >0$. One can select $\beta^k \equiv \beta$ with $\beta \leq \frac{\nu}{L}$. However, in some applications, $L$ can not be explicitly computed, or it is evaluated	conservatively, the latter causes small $\beta$ which usually results in slow convergence. Thus, one can take some adaptive rules for choosing $\beta^k$ at each iteration.
\end{remark}
\section{Proximity and Contraction Algorithms for $\text{MGVI}(\theta,F)$}
\label{sec:4}
In this section, we will introduce the second class of our PGA-based Predictor-Corrector algorithms for $\text{MGVI}(\theta,F)$. We are inspired by the works of He for $\text{MVI}(\theta,F,\C)$ in  \cite{MR1418264,MR1297058}, where the projection operator is used for making a predictor, while our algorithms for $\text{MGVI}(\theta,F)$ use proximity operator for making a predictor. The corrector strategies developed by He for $\text{MVI}(\theta,F,\C)$ can be inherited in our algorithms, this is an inspiring result. Considering that the algorithms developed by He for $\text{MVI}(\theta,F,\C)$ are called as projection and contraction algorithms, our algorithms depend on the strategies for corrector steps by He, and our algorithms also enjoy the contraction properties, thus we name our algorithms as proximity and contraction algorithms. In the sequel, we will just employ part of the corrector strategies developed by He for constructing our algorithms, the reader can refer to \cite{MR1418264,MR1297058} for more strategies for corrector step.

Firstly, we introduce three fundamental inequalities, which are key for constructing our  algorithms for $\text{MGVI}(\theta,F)$.  
\begin{proposition}[Three fundamental inequalities]
	Let $x^*$ be any solution of $\text{MGVI}(\theta,F)$, and for any $x\in \R^n$ and $\beta >0$, let $\tilde{x} = \text{Prox}_{\beta \theta}(x-\beta F(x))$, then the following three inequalities hold:
	\begin{equation*}
	\label{FI1}
	\beta\theta(\tilde{x}) - \beta \theta(x^*) + (\tilde{x}-x^*)^T\beta F(x^*) \geq 0, \tag{FI1}
	\end{equation*}
	\begin{equation*}
	\label{FI2}
	\beta \theta(x^*) -\beta\theta(\tilde{x}) + (\tilde{x}-x^*)^T([x-\beta F(x)] -\tilde{x}) \geq 0,
	\tag{FI2}
	\end{equation*}
	\begin{align*}
	\label{FI3}
	(\tilde{x}-x^*)^T(\beta F(\tilde{x})-\beta F(x^*)) \geq 0. 
	\tag{FI3}
	\end{align*}
\end{proposition}
\begin{proof}
	The first inequality (\ref{FI1}) follows from the definition of $\text{MGVI}(\theta,F)$. The final inequality (\ref{FI3}) holds because $F$ is monotone over $\R^n$. Now let us prove the second inequality (\ref{FI2}). Because $\tilde{x} = \text{Prox}_{\beta \theta}(x-\beta F(x))$, then it follows from Lemma \ref{lem:1} that for all $y\in \R^n$, it holds that
	\begin{equation*}
	\beta \theta(y) \geq \beta \theta(\tilde{x}) + (y-\tilde{x})^T(x-\beta F(x)-\tilde{x}),
	\end{equation*}
	replacing $y$ by $x^*$, and by rearrangement, then we conclude (\ref{FI2}).
	\qed
\end{proof}

We observe that the leftmost terms of (\ref{FI1}) and (\ref{FI2}) are opposite to each other, so they can be eliminated when adding (\ref{FI1}) and (\ref{FI2}), or when adding (\ref{FI1}), (\ref{FI2}) and (\ref{FI3}). This observance makes it possible for us to directly employ the strategies of corrector step developed by He in \cite{MR1418264,MR1297058} for $\text{MVI}(\theta,F,\C)$.

\subsection{Proximity and contraction algorithms based on (\ref{FI1}) and (\ref{FI2})}

In this part, we consider $\text{MGVI}(\theta,F)$ with $F(x) = Mx + q$. Then (\ref{FI1}) and (\ref{FI2}) take the next forms
\begin{equation*}
\beta\theta(\tilde{x}) - \beta \theta(x^*) + (\tilde{x}-x^*)^T\beta (Mx^*+q) \geq 0,
\end{equation*}
and 
\begin{equation*}
\beta \theta(x^*) -\beta\theta(\tilde{x}) + (\tilde{x}-x^*)^T([x-\beta(Mx+q)] -\tilde{x}) \geq 0.
\end{equation*}
Adding the above two inequalities together, one obtains that
\begin{align}
\label{ine:1}
(x-x^*)^T(I+\beta M^T)(x-\tilde{x}) &\geq (x-x^*)^T\beta M(x-x^*) + \Vert x-\tilde{x}\Vert^2 \nonumber\\
& \geq \Vert x-\tilde{x}\Vert^2,
\end{align}
where the last inequality holds because $M$ is positive semi-definite\footnote{Here, $M$ is not required to be symmetric.}, i.e., for all $x,y\in\R^n$, $(x-y)^TM(x-y) \geq 0$. Replacing the above $x$, $\tilde{x}$, $\beta$ by $x^k$, $\tilde{x}^k$, $\beta^k$ respectively, and setting
\begin{align*}
\phi(x^k,\tilde{x}^k,\beta^k) &:= \Vert x^k-\tilde{x}^k\Vert^2, \\
d(x^k,\tilde{x}^k,\beta^k) &:= (I+\beta^k M^T)(x^k-\tilde{x}^k),
\end{align*}  
then it follows that $(x^k-x^*)^Td(x^k,\tilde{x}^k,\beta^k) \geq \phi(x^k,\tilde{x}^k,\beta^k)$. Next, we use $$x^{k+1} = x^k - \alpha d(x^k,\tilde{x}^k,\beta^k)$$ to obtain the next surrogate, then one sees
\begin{align*}
\vartheta_k(\alpha) &:= \Vert x^k -x^*\Vert^2 - \Vert x^{k+1} - x^*\Vert^2 \\
&= 2\alpha (x^k-x^*)^Td(x^k,\tilde{x}^k,\beta^k) - \alpha^2\Vert d(x^k,\tilde{x}^k,\beta^k)\Vert^2 \\
&\geq 2\alpha \phi(x^k,\tilde{x}^k,\beta^k) - \alpha^2 \Vert d(x^k,\tilde{x}^k,\beta^k)\Vert^2 := q_k(\alpha).
\end{align*}
Note that $q_k(\alpha)$ obtains its maximum at $\alpha_k^* = \frac{\phi(x^k,\tilde{x}^k,\beta^k)}{\Vert d(x^k,\tilde{x}^k,\beta^k)\Vert^2}$, thus by setting $x^{k+1} = x^k - \gamma \alpha^*d(x^k,\tilde{x}^k,\beta^k)$, one sees that 
\begin{align}
\label{PGAa1}
\Vert x^{k+1} -x^*\Vert^2 &\leq \Vert x^k-x^*\Vert^2 - \gamma(2-\gamma)\alpha_k^*\phi(x^k,\tilde{x}^k,\beta^k) \nonumber\\
&= \Vert x^k-x^*\Vert^2 - \gamma(2-\gamma)\alpha_k^*\Vert x^k -\tilde{x}^k\Vert^2.
\end{align}
Note that $\alpha_k^* = \frac{\phi(x^k,\tilde{x}^k,\beta^k)}{d(x^k,\tilde{x}^k,\beta^k)} \geq \frac{1}{\Vert I+\beta^k M^T\Vert_2}$. Thus, $\inf\{\beta^k\}_k > 0$ implies that  $\inf \{\alpha_k^*\}_k > 0$. 

The aforementioned algorithm, denoted as $\text{PGA}_{a_1}$, is summarized in Algorithm \ref{algo:PGA1}.

\begin{algorithm}[ht]
	\caption{$\text{PGA}_{a_1}$ for $\text{MGVI}(\theta,F)$}
	\label{algo:PGA1}
	\KwIn{$x^0 \in \R^n$, and $\gamma \in (0,2)$.}
	
	\For{$k=0, 1, \cdots$}
	{   
		\textbf{Predictor:}  Selecting $\beta_k >0$, and $\tilde{x}^k = \text{Prox}_{\beta^k \theta}(x^k-\beta^k F(x^k))$ ;
		
		\textbf{Corrector:} Setting $d(x^k,\tilde{x}^k,\beta^k) = (I+\beta^kM^T)(x^k-\tilde{x}^k)$, and computing $\alpha_k^* = \frac{\Vert x^k-\tilde{x}^k\Vert^2}{\Vert d(x^k,\tilde{x}^k,\beta^k)\Vert^2}$.
		
		Setting $x^{k+1} = x^k - \gamma\alpha_k^*d(x^k,\tilde{x}^k,\beta^k)$.
	}
	
\end{algorithm}

Next, we construct another proximity and contraction algorithm, which is based on that $M$ is symmetric positive semi-definite. In this case, setting $\phi(x^k,\tilde{x}^k) := \Vert x^k-\tilde{x}^k\Vert^2$, $d(x^k,\tilde{x}^k) :=  x^k-\tilde{x}^k$, and let $G:= I+\beta M^T$, $x^{k+1} = x^k -\alpha d(x^k,\tilde{x}^k)$, then one sees $(x^k-x^*)^TGd(x^k,\tilde{x}^k) \geq \phi(x^k,\tilde{x}^k)$. Let $\vartheta_k(\alpha) := \Vert x^k-x^*\Vert_G^2 - \Vert x^{k+1} -x^*\Vert_G^2$, one sees that 
\begin{align*}
\vartheta_k(\alpha) &= 2(x^k-x^*)Gd(x^k,\tilde{x}^k) - \alpha^2\Vert d(x^k,\tilde{x}^k)\Vert_G^2\\
&\geq 2\phi(x^k,\tilde{x}^k) -\alpha^2\Vert d(x^k,\tilde{x}^k)\Vert_G^2 := q_k(\alpha).
\end{align*}
Note that $q_k(\alpha)$ obtains its maximum at $\alpha_k^* = \frac{\phi(x^k,\tilde{x}^k)}{\Vert d(x^k,\tilde{x}^k)\Vert_G^2}$, thus by setting $x^{k+1} = x^k - \gamma \alpha_k^*d(x^k,\tilde{x}^k)$, one sees that 
\begin{align*}
\Vert x^{k+1} -x^*\Vert_G^2 &\leq \Vert x^k-x^*\Vert_G^2 - \gamma(2-\gamma)\alpha_k^*\phi(x^k,\tilde{x}^k,\beta^k) \\
&= \Vert x^k-x^*\Vert_G^2 - \gamma(2-\gamma)\alpha_k^*\Vert x^k -\tilde{x}^k\Vert^2.
\end{align*} 
Note that $\alpha_k^* = \frac{\phi(x^k,\tilde{x}^k)}{\Vert d(x^k,\tilde{x}^k)\Vert_G^2} \geq \frac{1}{\lambda_{\text{max}}(G)}$.

The above method for symmetric $M$, denoted as $\text{PGA}_{a_2}$, is summarized below in Algorithm \ref{algo:PGA2}.
\begin{algorithm}[ht]
	\caption{$\text{PGA}_{a_2}$ for $\text{MGVI}(\theta,F)$}
	\label{algo:PGA2}
	\KwIn{$x^0 \in \R^n$, $\beta >0$, and $\gamma \in (0,2)$.}
	
	\For{$k=0, 1, \cdots$}
	{   
		\textbf{Predictor:}  $\tilde{x}^k = \text{Prox}_{\beta \theta}(x^k-\beta F(x^k))$ ;
		
		\textbf{Corrector:} Setting $d(x^k,\tilde{x}^k) = x^k-\tilde{x}^k$, and computing $\alpha_k^* = \frac{\Vert x^k-\tilde{x}^k\Vert^2}{\Vert d(x^k,\tilde{x}^k)\Vert_{G}^2}$,
		
		Setting $x^{k+1} = x^k - \gamma\alpha_k^*d(x^k,\tilde{x}^k)$.
	}
	
\end{algorithm}
\begin{theorem}[Convergence of $\text{PGA}_{a_1}$ and $\text{PGA}_{a_2}$ for $\text{MGVI}(\theta,F)$]
	\label{them:2}
	Let $\{x^k\}_k$ and $\{y^k\}_k$ be the sequences generated by $\text{PGA}_{a_1}$ and $\text{PGA}_{a_2}$, respectively, and for $\text{PGA}_{a_1}$, let $\{\beta^k\}_k$ be bounded with $\inf \{\beta_k\} >0$. Then $\{x^k\}_k$ (resp. $\{y^k\}_k)$ converges to a solution  of the corresponding $\text{MGVI}(\theta,F)$.
\end{theorem}
\begin{proof}
	We will just prove the convergence of $\text{PGA}_{a_1}$ since the proof of the convergence of $\text{PGA}_{a_2}$ follows the similar argument.
	
	Let $x^*\in \Omega^*$ be any solution of $\text{MGVI}(\theta,F)$. The inequality (\ref{PGAa1}) implies that $\{x^k\}_k$ is Fej\'{e}r monotone with respect to $\Omega^*$, moreover, by adding from $k = 0$ to $n$, taking some rearrangements, and letting $n\rightarrow \infty$, then it holds that 
	\begin{equation*}
	\sum_{k=0}^{\infty}\gamma(2-\gamma)\alpha_k^*\Vert x^k-\tilde{x}^k\Vert^2 \leq \Vert x_0 - x^*\Vert^2.
	\end{equation*} 
	Because $\inf\{\alpha_k^*\}>0$, thus the above inequality implies $\Vert x^k - \tilde{x}^k\Vert \rightarrow 0$, for which we conclude that  $\{x^k\}_k$ and $\{\tilde{x}^k\}_k$ have common limit points. Then the similar argument in Theorem \ref{them:1} implies that  $\{x^k\}_k$ converges to a solution of $\text{MGVI}(\theta,F)$. \qed
\end{proof}

\begin{remark}
	To guarantee that $\text{PGA}_{a_1}$ (Algorithm \ref{algo:PGA1}) converges, $\{\beta^k\}_k$ requires to be bounded such that $\inf \{\beta^k\}>0$. The adaptive rule $$\Vert \beta^kM^T(x^k-\tilde{x}^k)\Vert = \mathcal{O}(\Vert x^k-\tilde{x}^k\Vert)$$ was suggested by He for $\text{MVI}(\theta,F,\C)$, and can also be applied to our algorithms. For $\text{PGA}_{a_2}$, $\beta^k \equiv \beta$, the convergence is guaranteed with $\beta$ being any positive constant, one possibility may be  $\beta\Vert M\Vert_2 \approx 1$.  
\end{remark}
\subsection{Proximity and contraction algorithms based on (\ref{FI2}), (\ref{FI2}) and (\ref{FI3})}
In this part, we consider the case where $F$ is monotone. By adding (\ref{FI1}), (\ref{FI2}) and (\ref{FI3}) together, one obtains that 
\begin{equation*}
(\tilde{x}-x^*)^T\{(x-\tilde{x}) -\beta(F(x)-F(\tilde{x})) \} \geq 0,
\end{equation*}
thus it holds that 
\begin{equation}
\label{fi2_based1}
(x-x^*)^T\{(x-\tilde{x})-\beta(F(x)-F(\tilde{x}))\} \geq (x-\tilde{x})^T\{(x-\tilde{x})-\beta(F(x)-F(\tilde{x}))\}.
\end{equation}
Replacing the above $x$, $\tilde{x}$, $\beta$ by $x^k$, $\tilde{x}^k$, $\beta^k$ respectively, and setting
\begin{align*}
\phi(x^k,\tilde{x}^k,\beta^k) &:= (x^k-\tilde{x}^k)^T\{(x^k-\tilde{x}^k)-\beta^k(F(x^k)-F(\tilde{x}^k))\}, \\
d(x^k,\tilde{x}^k,\beta^k) &:= (x^k-\tilde{x}^k)-\beta^k(F(x^k)-F(\tilde{x}^k)),
\end{align*} 
then one sees that $(x^k-x^*)^Td(x^k,\tilde{x}^k,\beta^k) \geq \phi(x^k,\tilde{x}^k,\beta^k)$. For a fixed $\nu\in(0,1)$, considering that $F$ is lipschitz continuous with parameter $L>0$, thus some $\beta^k$ can be selected such that 
\begin{equation*}
\beta^k \Vert F(x^k) - F(\tilde{x}^k) \Vert \leq \nu \Vert x^k-\tilde{x}^k\Vert,  
\end{equation*}
thus $\phi(x^k,\tilde{x}^k,\beta^k) \geq (1-\nu)\Vert x^k-\tilde{x}^k\Vert^2$. By setting $x^{k+1} = x^k - \alpha d(x^k,\tilde{x}^k,\beta^k)$, one sees that 
\begin{align*}
\vartheta_k(\alpha) &:= \Vert x^k - x^*\Vert^2 - \Vert x^{k+1} - x^*\Vert^2 \\
&= 2\alpha(x^k-x^*)^Td(x^k,\tilde{x}^k,\beta^k) - \alpha^2\Vert d(x^k,\tilde{x}^k,\beta^k)\Vert^2\\
&\geq 2\alpha \phi(x^k,\tilde{x}^k,\beta^k) - \alpha^2\Vert d(x^k,\tilde{x}^k,\beta^k)\Vert^2 := q_k(\alpha).
\end{align*}
Note that $q_k(\alpha)$ obtains its maximum at $\alpha_k^* = \frac{\phi(x^k,\tilde{x}^k,\beta^k)}{\Vert d(x^k,\tilde{x}^k,\beta^k)\Vert^2}$, and one also sees that 
\begin{align*}
2\phi(x^k,\tilde{x}^k,\beta^k) - \Vert d(x^k,\tilde{x}^k,\beta^k)\Vert^2 &= 2(x^k-\tilde{x}^k)^Td(x^k,\tilde{x}^k,\beta^k)-\Vert d(x^k,\tilde{x}^k,\beta^k)\Vert^2\\
&= d(x^k,\tilde{x}^k,\beta^k)^T\{2(x^k-\tilde{x}^k)-d(x^k,\tilde{x}^k,\beta^k)\} \\
&= (x^k-\tilde{x}^k)^T - {\beta^k}^2\Vert F(x^k)-F(\tilde{x}^k)\Vert^2 \\
&\geq (1-\nu^2)\Vert x^k-\tilde{x}^k\Vert^2,
\end{align*}
thus $\alpha_k^* > \frac{1}{2}$ for all $k$. Setting $x^{k+1} = x^k - \gamma \alpha_k^* d(x^k,\tilde{x}^k,\beta^k)$, then one obtains that
\begin{align*}
\Vert x^{k+1} -x^*\Vert^2  &\leq \Vert x^k -x^* \Vert^2- (2-\gamma)\gamma \alpha_k^*\phi(x^k,\tilde{x}^k,\beta^k) \\ 
& \leq \Vert x^k -x^*\Vert^2 - (2-\gamma)\gamma\alpha_k^*(1-\nu)\Vert x^k-\tilde{x}^k\Vert^2.
\end{align*}

The aforementioned algorithm, denoted as $\text{PGA}_{b_1}$, is summarized in Algorithm \ref{algo:PGAb1}.

\begin{algorithm}[ht]
	\caption{$\text{PGA}_{b_1}$ for $\text{MGVI}(\theta,F)$}
	\label{algo:PGAb1}
	\KwIn{$x^0 \in \R^n$, and $\gamma \in (0,2)$.}
	
	\For{$k=0, 1, \cdots$}
	{   
		\textbf{Predictor:}  Selecting $\beta^k >0$ such that $\beta^k\Vert F(x^k)-F(\tilde{x}^k)\Vert \leq \nu \Vert x^k-\tilde{x}^k\Vert$ where $\tilde{x}^k = \text{Prox}_{\beta^k \theta}(x^k-\beta^k F(x^k))$ ;
		
		\textbf{Corrector:} Setting $d(x^k,\tilde{x}^k,\beta^k) = (x^k-\tilde{x}^k)-\beta^k(F(x^k)-F(\tilde{x}^k))$, and computing $\alpha_k^* = \frac{(x^k-\tilde{x}^k)^Td(x^k,\tilde{x}^k,\beta^k)}{\Vert d(x^k,\tilde{x}^k,\beta^k)\Vert^2}$,
		
		Setting $x^{k+1} = x^k - \gamma\alpha_k^*d(x^k,\tilde{x}^k,\beta^k)$.
	}
	
\end{algorithm}

Finally, we distinguish a special case where  $F(x) = Ax + c$ with $A$ being a symmetric positive semi-definite matrix, and $c\in \R^n$ being a constant vector. In this case, inequality (\ref{fi2_based1}) takes the next form
\begin{equation*}
(x-x^*)^T(I-\beta A)(x-\tilde{x}) \geq (x-\tilde{x})^T(I-\beta A)(x-\tilde{x}).
\end{equation*} 
Setting $G:=I-\beta A$ with $\beta$ being a fixed positive constant smaller than $\frac{1}{\lambda_{\max}(A)}$, thus $G$ is a symmetric positive definite matrix. Replacing the above $x$ and $\tilde{x}$ by $x^k$ and $\tilde{x}^k$, respectively. Then, one sees that
\begin{equation*}
(x^k-x^*)^TG(x^k-\tilde{x}^k) \geq (x^k-\tilde{x}^k)^TG(x^k-\tilde{x}^k).
\end{equation*}
Let $x^{k+1} = x^k - \gamma(x^k -\tilde{x}^k)$, thus it follows that 
\begin{equation*}
\Vert x^{k+1} -x^*\Vert_{G}^2 \leq \Vert x^{k} -x^*\Vert_{G}^2 - (2-\gamma)\gamma \Vert x^{k} -\tilde{x}^k\Vert_{G}^2.
\end{equation*}

The above method for symmetric $A$ is denoted as $\text{PGA}_{b2}$, the corrector strategy was not mentioned by He, but it can be simply constructed from (\ref{fi2_based1}), we  summarize it below in Algorithm \ref{algo:PGAb2}.
\begin{algorithm}[ht]
	\caption{$\text{PGA}_{b_2}$ for $\text{MGVI}(\theta,F)$}
	\label{algo:PGAb2}
	\KwIn{$x^0 \in \R^n$, $\beta >0$, and $\gamma \in (0,2)$.}
	
	\For{$k=0, 1, \cdots$}
	{   
		\textbf{Predictor:}  $\tilde{x}^k = \text{Prox}_{\beta \theta}(x^k-\beta F(x^k))$ ;
		
		\textbf{Corrector:} Setting $d(x^k,\tilde{x}^k) = x^k-\tilde{x}^k$,
		
		Setting $x^{k+1} = x^k - \gamma d(x^k,\tilde{x}^k)$.
	}
	
\end{algorithm}

The convergence analysis of $\text{PGA}_{b_1}$ and $\text{PGA}_{b_2}$ can be proved with essentially the argument that was used in Theorem \ref{them:2}. 
\begin{theorem}[Convergence of $\text{PGA}_{b_1}$ and $\text{PGA}_{b_2}$ for $\text{MGVI}(\theta,F)$]
	Let $\{x^k\}_k$ and $\{y^k\}_k$ be the sequences generated by $PGA_{b_1}$ and $PGA_{b_2}$, respectively, and for $PGA_{b_1}$, let $\{\beta^k\}_k$ be bounded with $\inf \{\beta_k\} >0$. Then $\{x^k\}_k$ (resp. $\{y^k\}_k)$ converges to a solution  of the corresponding $\text{MGVI}(\theta,F)$. 
\end{theorem}

\begin{remark}
	In the description of $\text{PGA}_{a_1}$, $\text{PGA}_{a_2}$, and $\text{PGA}_{b_1}$, we neglect the fact that $x^k$ may be equal to $\tilde{x}^k$, however, this will hardly happen in practice, moreover, when implementing these algorithms, one may priorly determine whether $x^k = \tilde{x}^k$.  
\end{remark}

\section{A Discussion about Our Algorithms}
\label{sec:5}
In the preceding sections, we introduce our algorithms for $\text{MGVI}(\theta,F)$, while in this part, we will show the characteristics of our algorithms with respect to the following separable two-blocks convex optimization model. 
\begin{equation}
\label{ADMM}
\begin{split}
\min\quad &\theta_1(x) + \theta_2(y) \\
\text{s.t.}\quad  & Ax + By =c, \\
\end{split}
\end{equation}
where $A\in\R^{m\times n}$, $B\in \R^{m\times q}$, $c\in \R^m$, $\theta_1:\R^n \rightarrow (-\infty,\infty]$ and $\theta_2:\R^q \rightarrow (-\infty,\infty]$ are two closed proper convex functions. 

The well known alternating direction method of
multipliers (ADMM) \cite{MR724072} is very popular for problem (\ref{ADMM}), the reader can refer to \cite{BoydADMM} and references therein for efficient applications of ADMM in a variety of fields. The procedure of ADMM is shown below in Algorithm \ref{algo:ADMM}. 
\begin{algorithm}[ht]
	\caption{ADMM for problem (\ref{ADMM})}
	\label{algo:ADMM}
	\textbf{Initialization:} $x^0 \in \R^n$, $y^0\in\R^q$, $\lambda^0\in\R^m$, and $\rho>0$.
	
	\textbf{General step:} for $k=0,1,\cdots$ execute the following:
	\begin{enumerate}
		\item[(a)] $x^{k+1} \in \argmin\{\theta_1(x)+\frac{\rho}{2}\Vert Ax+By^k-c+\frac{1}{\rho}\lambda^k\Vert^2:x\in\R^n\}$;
		\item[(b)] $y^{k+1} \in \argmin \{\theta_2(y)+\frac{\rho}{2}\Vert Ax^{k+1}+By-c+\frac{1}{\rho}\lambda^k\Vert^2:y\in\R^q\}$;
		\item[(c)] $\lambda^{k+1} = \lambda^k +\rho(Ax^{k+1}+By^{k+1}-c)$.
	\end{enumerate}	
\end{algorithm}
Note that the computation of $x^{k+1}$ and $y^{k+1}$ in the general step of ADMM may be much more complex than the computation of the proximity operators of $\theta_1$ and $\theta_2$, thus some variants of ADMM have been proposed for dealing with such matters. One of them is the alternating direction proximal method of multipliers (AD-PMM) \cite{MR3719240}, and a special case of AD-PMM is the the alternating direction linearized proximal method
of multipliers (AD-LPMM), which just requires the computation of the proximity operators of $\theta_1$ and $\theta_2$. AD-LPMM for problem (\ref{ADMM}) is show in Algorithm \ref{algo:AD-LPMM}.
\begin{algorithm}[ht]
	\caption{AD-LPMM for problem (\ref{ADMM})}
	\label{algo:AD-LPMM}
	\textbf{Initialization:} $x^0 \in \R^n$, $y^0\in\R^q$, $\lambda^0\in\R^m$, $\rho>0$, $\alpha\geq \rho \lambda_{\max}(A^TA)$, $\beta \geq \rho \lambda_{\max}(B^TB)$.
	
	\textbf{General step:} for $k=0,1,\cdots$ execute the following:
	\begin{enumerate}
		\item[(a)] $x^{k+1} =\text{Prox}_{\frac{1}{\alpha}\theta_1}[x^k-\frac{\rho}{\alpha}A^T(Ax^k+By^k-c+\frac{1}{\rho}\lambda^k)]$;
		\item[(b)] $y^{k+1} =\text{Prox}_{\frac{1}{\beta}\theta_2}[y^k-\frac{\rho}{\beta}B^T(Ax^{k+1}+By^k-c+\frac{1}{\rho}\lambda^k)]$;
		\item[(c)] $\lambda^{k+1} = \lambda^k +\rho(Ax^{k+1}+By^{k+1}-c)$.
	\end{enumerate}	
\end{algorithm}
 
Next, we will introduce our algorithms for problem (\ref{ADMM}).  Consider the lagrangian function of problem (\ref{ADMM}) $$L(x,y,\lambda) = \theta_1(x)+\theta_2(y) -\lambda^T(Ax+By-c).$$ Then $(x^*,y^*,\lambda^*)\in \R^n\times\R^q\times\R^m$ is a saddle point of problem (\ref{ADMM}) if and only if 
$$\min_{\lambda}L(x^*,y^*,\lambda)\leq L(x^*,y^*,\lambda^*) \leq \min_{x,y} L(x,y,\lambda^*),$$
the latter is equivalent to 
\begin{equation*}
\Big[\begin{array}{c}
\theta_1(x)	\\ 
+\theta_2(y)	
\end{array}\Big]
-\Big[\begin{array}{c}
\theta_1(x^*)	\\ 
+\theta_2(y^*)	
\end{array}\Big] + \Bigg[\begin{pmatrix}
x\\ 
y\\ 
\lambda
\end{pmatrix} - \begin{pmatrix}
x^*\\ 
y^*\\ 
\lambda^*
\end{pmatrix}\Bigg]^T
\Bigg[
\begin{pmatrix}
0&0  & -A^T \\ 
0&0  & -B^T \\ 
A& B & 0
\end{pmatrix} \begin{pmatrix}
x\\ 
y\\ 
\lambda
\end{pmatrix}+
\begin{pmatrix}
0\\ 
0\\ 
-c
\end{pmatrix}\Bigg] \geq 0,
\end{equation*} 
for all $(x,y,\lambda)\in \R^n \times \R^q \times \R^m$. This fits $\text{MGVI}(\theta,F)$ with
\begin{equation*}
\begin{array}{ll}
\theta(x,y,\lambda) &= \theta_1(x) +\theta_2(y),  \\
F(x,y,\lambda) &= \begin{pmatrix}
0&0  & -A^T \\ 
0&0  & -B^T \\ 
A& B & 0
\end{pmatrix} \begin{pmatrix}
x\\ 
y\\ 
\lambda
\end{pmatrix}+
\begin{pmatrix}
0\\ 
0\\ 
-c
\end{pmatrix}.
\end{array}
\end{equation*} 
For the theory of saddle point, one can see e.g., \cite{MR2184037}. Now we only need to keep in mind that if $(x^*,y^*,\lambda^*)$ is a saddle point of problem (\ref{ADMM}), then $(x^*,y^*)$ will be a minimizer of problem (\ref{ADMM}).

Note that our algorithms $\text{PGA}_{a_2}$ and $\text{PGA}_{b_2}$ are designed for some specified operators about $F$, which can not be applied for this problem. Moreover, for problem (\ref{ADMM}), owing to the special form of $F$, this causes that the resulting $\text{PGA}_{a_1}$ and $\text{PGA}_{b_1}$ will only differ in the computation of $\alpha_k^*$ when the same self-adaptive rule is employed. For these reasons, we will firstly show GEM and $\text{PGA}_{a_1}$ in Algorithm \ref{algo:GEM-ADMM} and \ref{algo:PGA1-ADMM}, respectively. A self-adaptive rule for $\{\beta^k\}_k$ is also included into these two algorithms. The computation of $\alpha_k^*$ in $\text{PGA}_{b_1}$ for problem (\ref{ADMM}) is given in Algorithm \ref{algo:PGAb1-ADMM}.

\begin{algorithm}[ht]
	\caption{GEM for problem (\ref{ADMM})}
	\label{algo:GEM-ADMM}
	\textbf{Initialization:} $x^0 \in \R^n$, $y^0\in\R^q$, $\lambda^0\in\R^m$, $\beta^0 >0$, and $\nu,\mu\in(0,1)$ with $\mu < \nu$.
	
	\For{$k=0, 1, \cdots$}
	{   
		
		$\tilde{x}^k = \text{Prox}_{\beta^k\theta_1}(x^k+\beta^kA^T\lambda^k);$ 
		
		$\tilde{y}^k = \text{Prox}_{\beta^k \theta_2}(y^k+\beta^kB^T\lambda^k)$;
		
		$\tilde{\lambda}^k = \lambda^k -\beta^k(Ax^k+By^k-c)$;
		
		$r_k = \beta^k\Bigg{\Vert}
		\begin{pmatrix}
		A^T(\lambda^k -\tilde{\lambda}^k)\\		B^T(\lambda^k-\tilde{\lambda}^k)\\ 
		A(x^k-\tilde{x}^k)+B(y^k-\tilde{y}^k)
		\end{pmatrix}\Bigg{\Vert}\Bigg{/}
		\Bigg{\Vert}\begin{pmatrix}
		x^k-\tilde{x}^k\\
		y^k-\tilde{y}^k\\ 
		\lambda^k-\tilde{\lambda}^k\\ 
		\end{pmatrix}\Bigg{\Vert}; $


		\SetKwIF{If}{ElseIf}{Else}{if}{then}{else if}{else}{end}
		\If{$r_k>\nu$}
		{
			$\beta^k := \frac{2}{3}*\beta^k\min\{1,\frac{1}{r_k}\}$;
			
			\textbf{Goto} line 3.
		}
		
		$x^{k+1} = \text{Prox}_{\beta^k\theta_1}(x^k+\beta^kA^T\tilde{\lambda}^k);$ 
		
		$y^{k+1} = \text{Prox}_{\beta^k \theta_2}(y^k+\beta^kB^T\tilde{\lambda}^k)$;
		
		$\lambda^{k+1} = \lambda^k -\beta^k(A\tilde{x}^k+B\tilde{y}^k-c)$;
		
		\SetKwIF{If}{ElseIf}{Else}{if}{then}{else if}{else}{end}
		\eIf{$r_k\leq \mu$}
		{
			$\beta^{k+1} := 1.5*\beta^k$;
		}
		{	
			$\beta^{k+1} := \beta^k$.
		}
		
	}
\end{algorithm}

\begin{algorithm}[ht]
	\caption{$\text{PGA}_{a_1}$ for problem (\ref{ADMM})}
	\label{algo:PGA1-ADMM}
	\textbf{Initialization:} $x^0 \in \R^n$, $y^0\in\R^q$, $\lambda^0\in\R^m$, $\beta^0 >0$, $\nu \in (0,1)$, and $\gamma \in (0,2)$.
	
	\For{$k=0, 1, \cdots$}
	{     
		$\tilde{x}^k = \text{Prox}_{\beta^k\theta_1}(x^k+\beta^kA^T\lambda^k);$
		 
		$\tilde{y}^k = \text{Prox}_{\beta^k \theta_2}(y^k+\beta^kB^T\lambda^k);$
		
		$\tilde{\lambda}^k = \lambda^k -\beta^k(Ax^k+By^k-c);$
		
		$r_k = \beta^k\Bigg{\Vert}
		\begin{pmatrix}
		A^T(\lambda^k -\tilde{\lambda}^k)\\		B^T(\lambda^k-\tilde{\lambda}^k)\\ 
		A(x^k-\tilde{x}^k)+B(y^k-\tilde{y}^k)
		\end{pmatrix}\Bigg{\Vert}\Bigg{/}
		\Bigg{\Vert}\begin{pmatrix}
		x^k-\tilde{x}^k\\
		y^k-\tilde{y}^k\\ 
		\lambda^k-\tilde{\lambda}^k\\ 
		\end{pmatrix}\Bigg{\Vert}; $
		
		\If{$r_k>\nu$}
		{
			$\beta^k := \frac{2}{3}*\beta^k\min\{1,\frac{1}{r_k}\}$;
			
			\textbf{Goto} line 3.
		}
		
		$\alpha_k^* = \Bigg{\Vert} 
		\begin{pmatrix}
		x^k-\tilde{x}^k\\
		y^k-\tilde{y}^k\\
		\lambda^k-\tilde{\lambda}^k
			\end{pmatrix} \Bigg{\Vert}^2\Bigg{/}  \Bigg{\Vert}\begin{pmatrix}
				[x^k-\tilde{x}^k]+\beta^k[A^T(\lambda^k-\tilde{\lambda}^k)]\\ 
				[y^k-\tilde{y}^k]+\beta^k[B^T(\lambda^k-\tilde{\lambda}^k)]\\
				[\lambda^k-\tilde{\lambda}^k]-\beta^k[A(x^k-\tilde{x}^k)+B(y^k-\tilde{y}^k)]\\
			\end{pmatrix}\Bigg{\Vert}^2 $;
		
		$x^{k+1} = x^k-\gamma\alpha_k^*\left([x^k-\tilde{x}^k]+\beta^k[A^T(\lambda^k-\tilde{\lambda}^k)]\right)$;
		
		$y^{k+1} = y^k -\gamma\alpha_k^*\left([y^k-\tilde{y}^k]+\beta^k[B^T(\lambda^k-\tilde{\lambda}^k)]\right)$;
		
		$\lambda^{k+1} = \lambda^k - \gamma \alpha_k^*\left([\lambda^k-\tilde{\lambda}^k]-\beta^k[A(x^k-\tilde{x}^k)+B(y^k-\tilde{y}^k)]\right)$;
		
		\SetKwIF{If}{ElseIf}{Else}{if}{then}{else if}{else}{end}
		\eIf{$r_k\leq \mu$}
		{
			$\beta^{k+1} := 1.5*\beta^k$;
		}
		{	
			$\beta^{k+1} := \beta^k$.
		}
}
	
\end{algorithm} 

\begin{algorithm}[ht]
	\caption{$\text{PGA}_{b_1}$ for problem (\ref{ADMM})}
	\label{algo:PGAb1-ADMM}
	$\cdots$
	
	$\alpha_k^* = \frac{\begin{pmatrix}
		x^k-\tilde{x}^k	\\
		y^k-\tilde{y}^k \\
		\lambda^k-\tilde{\lambda}^k	\\ 
		
	\end{pmatrix}^T 
	\begin{pmatrix}
		[x^k-\tilde{x}^k]+[\beta^kA^T(\lambda^k-\tilde{\lambda}^k)]\\ 
		[y^k-\tilde{y}^k]+[\beta^kB^T(\lambda^k-\tilde{\lambda}^k)]\\
		[\lambda^k-\tilde{\lambda}^k]-\beta^k[A(x^k-\tilde{x}^k)+B(y^k-\tilde{y}^k)]\\
	\end{pmatrix}
}{ \Bigg{\Vert}\begin{pmatrix}
	[x^k-\tilde{x}^k]+[\beta^kA^T(\lambda^k-\tilde{\lambda}^k)]\\ 
	[y^k-\tilde{y}^k]+[\beta^kB^T(\lambda^k-\tilde{\lambda}^k)]\\
	[\lambda^k-\tilde{\lambda}^k]-\beta^k[A(x^k-\tilde{x}^k)+B(y^k-\tilde{y}^k)]\\
\end{pmatrix}\Bigg{\Vert}^2 }$;

	$\cdots$

\end{algorithm}

In the preceding part of this section, we have introduced five algorithms for problem (\ref{ADMM}). Now, let us take a close look at them. Firstly, ADMM (Algorithm \ref{algo:ADMM}) may not be implementable in practice because compared with the computation of proximity operators of $\theta_1$ and $\theta_2$, it may be much harder to solve the next two convex optimization problems
\begin{equation*}
\begin{array}{l}
\min\{\theta_1(x)+\frac{\rho}{2}\Vert Ax+By^k-c+\frac{1}{\rho}\Vert^2:x\in \R^n\},\\
\min\{\theta_2(y)+\frac{\rho}{2}\Vert Ax^{k+1}+By-c+\frac{1}{\rho}\Vert^2:y\in \R^n\}.
\end{array}
\end{equation*}
Next, let us look at AD-LPMM (Algorithm \ref{algo:AD-LPMM}) and the other three algorithms proposed by us. Firstly, at iteration $k$, AD-LPMM  solves $(x^{k+1},y^{k+1},\lambda^{k+1})$ in order, while the other three algorithms solve $(\tilde{x}^{k},\tilde{y}^{k}, \tilde{\lambda}^{k})$ and $(x^{k+1},y^{k+1},\lambda^{k+1})$ in parallel, this is a significant characteristic of our algorithms, allowing lesser time exhausted by employing the technique of parallel computation, especially in case of large-scale multi-blocks separable problems. Secondly, although the proximity operators of $\theta_1$ and $\theta_2$ are used by all of the four algorithms,  more computation is involved by AD-LPMM in updating $x$ and $y$ because it requires to compute $A^TAx$ and $A^TBy$, while our algorithms only require the computation of $A^T\lambda$ or $B^T\lambda$. Finally, AD-LPMM additionally evaluates in advance the maximal eigenvalues of $A^TA$ and $B^TB$, which is not necessary for our algorithms. Now, let us look at our proposed algorithms $\text{PGA}_{a_1}$, $\text{PGA}_{b_1}$, and GEM. The predictor steps are the same for all of them, who use the proximity operators of $\theta_1$ and $\theta_2$ for making the predictors. In corrector step, GEM employs again the proximity operators for obtaining the next surrograte $(x^{k+1},y^{k+1},\lambda^{k+1})$, while for the other two algorithms, only some cheap matrix-vector products, additions, and so on are involved. Thus, when the computation of proximity operators of $\theta_1$ and $\theta_2$ requires some iterative techniques, the corrector step of GEM will be more complex than that of $\text{PGA}_{a_1}$ and $\text{PGA}_{b_1}$. Finally, we point out that the value of $\alpha_k^*$ in $\text{PGA}_{b_1}$ is lesser or equal to its counterpart in $\text{PGA}_{a_1}$.
\section{Numerical Simulations}\label{sec:6}
In this section, we report some numerical results of our proposed algorithms. Our algorithms are implemented in MATLAB. The experiments are performed on a laptop equipped with Intel i$5$-$6200$U $2.30$GHz CPU and $8$ GB RAM. 

\subsection{First Example}
\label{sec:51}
Consider the next $l_1$-regularized least square problem
\begin{equation}
\min_{x\in\R^n} f(x):= \frac{1}{2}\Vert Ax-b\Vert^2 +\lambda \Vert x\Vert_1,
\end{equation}
where $A\in \R^{m\times n}$, $b\in\R^m$, and $\lambda >0$. This problem fits $\text{MGVI}(\theta,F)$ with 
\begin{equation*}
\begin{array}{l}
\theta(x) = \lambda \Vert x\Vert_1,\\
F(x) = A^T(Ax-b).
\end{array}
\end{equation*}
We generate an instance  of the  problem with $\lambda =1$ and $A\in\R^{1000\times 1100}$ for showing the  iterative shrinkage-thresholding algorithm (ISTA) and our proposed algorithms GME, $\text{PGA}_{a_1}$, $\text{PGA}_{a_2}$, $\text{PGA}_{b_1}$, $\text{PGA}_{b_2}$. The components of $A$ are independently generated by using the standard normal distribution, the "true" vector $\text{x\_{true}}$ is a sparse vector with 20 non-zeros elements, which is generated by the following matlab command:
	\begin{lstlisting}
		x_true = zeros(n,1);
		x_true(3:8:80) = 1;
		x_true(7:8:80) = -1;
	\end{lstlisting}
then we set $b:=A\text{x\_{true}}$. The initial point for all of the algorithms are $x=\textbf{e}$, the vector of all ones. The stopping criterion for all of them is $$\Vert x^k-\text{Prox}_{\theta}(x^k-F(x^k))\Vert_{\infty} < 10^{-6}.$$  

The trends about the values of $f$ of ISTA, GME, $\text{PGA}_{a_1}$, $\text{PGA}_{a_2}$, $\text{PGA}_{b_1}$, $\text{PGA}_{b_2}$ on $l_1$-regularized least squares problem are shown in Fig. \ref{fig:myplot3}. The number of iteration and CPU time for all of thm is shown in Table \ref{tab:1}.   All of them recover the "true" vector. 

\begin{figure}
	\centering
	\includegraphics[width=0.7\linewidth]{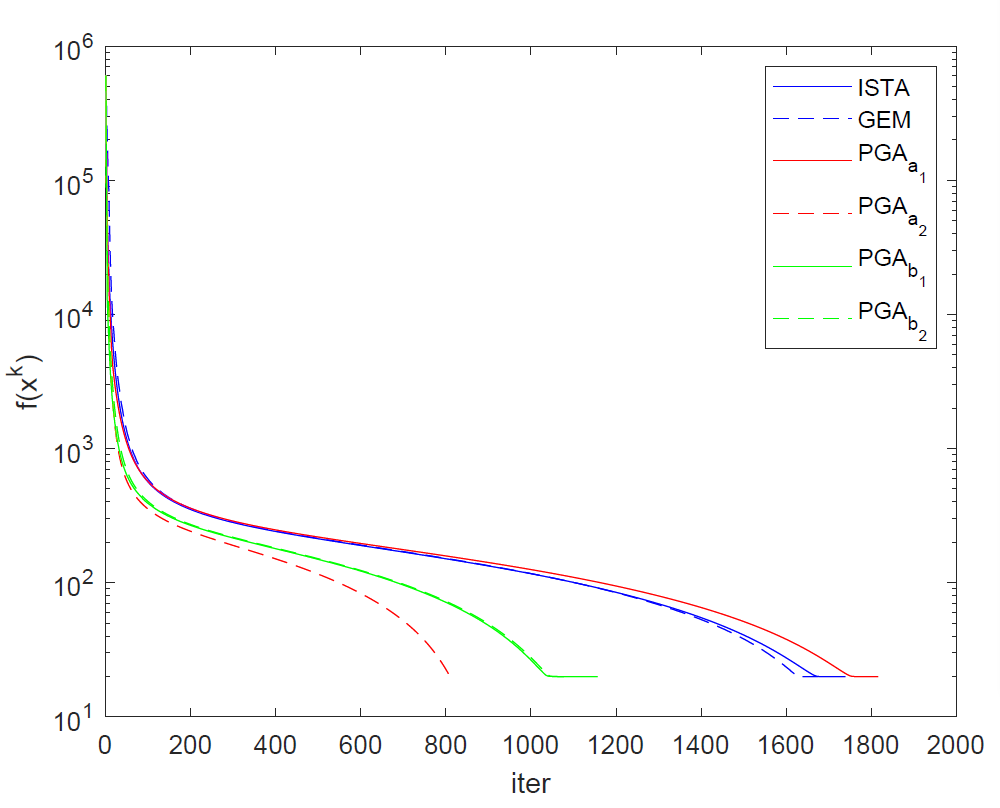}
	\caption{Trends about $f$ of ISTA, GME, $\text{PGA}_{a_1}$, $\text{PGA}_{a_2}$, $\text{PGA}_{b_1}$, $\text{PGA}_{b_2}$ on $l_1$-regularized least squares problem.}
	\label{fig:myplot3}
\end{figure}

\begin{table}[ht]
	\caption{The number of iteration and time of ISTA, GME, $\text{PGA}_{a_1}$, $\text{PGA}_{a_2}$, $\text{PGA}_{b_1}$, $\text{PGA}_{b_2}$ on $l_1$-regularized least squares problem.}
	\label{tab:1}       
	\begin{center}
	\begin{tabular}{ccc}
		\hline\noalign{\smallskip}
		 & No. of iteration & CPU time (seconds)  \\
		\noalign{\smallskip}\hline\noalign{\smallskip}
		ISTA               & 1739 & 5.36  \\
		GEM                & 1682 & 11.10 \\
		$\text{PGA}_{a_1}$ & 1816 & 24.93 \\
		$\text{PGA}_{a_2}$ & \textbf{822}  & 4.20  \\
		$\text{PGA}_{b_1}$ & 1157 & 7.31  \\
		$\text{PGA}_{b_2}$ & 1085 & \textbf{3.38}  \\
		\noalign{\smallskip}\hline
	\end{tabular}
	\end{center}
\end{table}

\subsection{Second Example}
Consider the next basis pursuit problem
\begin{equation}
\label{basis}
\begin{split}
\min\quad &\Vert x\Vert_1 \\
\text{s.t.}\quad  & Ax=b, \\
\end{split}
\end{equation}
where $A\in\R^{m\times n}$, $b\in\R^m$. This problem matches $\text{MGVI}(\theta,F)$ with $\theta(x,\lambda) = \Vert x\Vert_1$, and $F(x,\lambda) = \begin{pmatrix}
0& -A^T \\ 
A& 0
\end{pmatrix} 
\begin{pmatrix}
x\\
\lambda 
\end{pmatrix} +
\begin{pmatrix}
0\\
-b
\end{pmatrix}.
$
We use the same data as in the preceding example for showing AD-LPMM , and our algorithms GME, $\text{PGA}_{a_1}$, $\text{PGA}_{b_1}$. The initial point for all the algorithms are $x=\textbf{e}$, the vector of all ones, and $\lambda = \textbf{0}$, the vector of all zeros. The stopping criterion for all the algorithms except AD-LPMM are $\Vert (x^k,\lambda^k)-\text{Prox}_{\theta}\left((x^k,\lambda^k)-F(x^k,\lambda^k)\right)\Vert_{\infty} < 10^{-6}$, while that for AD-LPMM is $\Vert (x^{k+1},\lambda^{k+1})-(x^k,\lambda^k)\Vert_{\infty} < 10^{-6}$. Results of AD-LPMM, $\text{PGA}_{a_1}$, $\text{PGA}_{b_1}$, GEM on basis pursuit problem are shown in Table \ref{tab:2}, where we show the number of iteration and CPU time for all of them. All of these algorithms recover the "true" vector.

\begin{table}[ht]
	\caption{The number of iteration and time of AD-LPMM, GEM, $\text{PGA}_{a_1}$, $\text{PGA}_{b_1}$ on basis pursuit problem.}
	\label{tab:2}       
	\begin{center}
		\begin{tabular}{ccc}
			\hline\noalign{\smallskip}
			& No. of iteration & CPU time (seconds)  \\
			\noalign{\smallskip}\hline\noalign{\smallskip}
			AD-LPMM & 1773 & 18.65 \\
			GEM                & \textbf{105} & \textbf{4.65} \\
			$\text{PGA}_{a_1}$ & 225 & 7.58 \\
			$\text{PGA}_{b_1}$ & 226 & 7.05  \\
			\noalign{\smallskip}\hline
		\end{tabular}
	\end{center}
\end{table}

\section{Conclusion}
\label{sec:7}
In this paper, we propose some PGA-based Predictor-Corrector algorithms for solving $\text{MGVI}(\theta,F)$. Our works are inspired by the extragradient method and the projection and contraction algorithms for $\text{MVI}(\theta,F,\C)$. The form of our algorithms is also simple, and one of their significant characteristic for separable multi-blocks convex optimization problems is that they can be well adapted for parallel computation, allowing our algorithms can be well applied in large-scale settings. Preliminary simulations of our algorithms in the sparsity recovery models show their effectiveness when compared with some existing algorithms.

\begin{acknowledgements}
My special thanks should go to my friend, Kuan Lu, who has put considerable time and effort into his comments on the draft.
\end{acknowledgements}

\bibliographystyle{spmpsci}
\bibliography{mybib}   

%
%

\end{document}